\documentclass[a4paper,11pt]{article}
\usepackage[utf8]{inputenc}

\usepackage{amsmath, amssymb, amsthm}
\usepackage[english]{babel}
\usepackage{xcolor}
\usepackage[top=2.5cm,bottom=2.5cm,left=2.5cm,right=2.5cm]{geometry}
\usepackage{graphicx}
\usepackage{hyperref}
\usepackage{url}

\makeatletter
\g@addto@macro\bfseries{\boldmath}
\makeatother

\title{A higgledy-piggledy set of planes\\ based on the ABB-representation of linear sets}
\author{Lins Denaux \\ {\it Ghent University} \and Jozefien D'haeseleer \\ {\it Ghent University} \and Geertrui Van de Voorde \\ {\it University of Canterbury}}
\date{}

\newtheorem{thm}{Theorem}[section]
\newtheorem{lm}[thm]{Lemma}
\newtheorem{res}[thm]{Result}
\newtheorem{crl}[thm]{Corollary}
\newtheorem{prop}[thm]{Proposition}

\theoremstyle{definition}
\newtheorem{rmk}[thm]{Remark}
\newtheorem{df}[thm]{Definition}

\newcommand{\NN}{\mathbb{N}}
\newcommand{\FF}{\mathbb{F}}

\newcommand{\set}[1]{\left\{#1\right\}}
\newcommand{\sett}[2]{\left\{#1\,:\,#2\right\}}
\newcommand{\V}[2]{\mathrm{V}\!\mspace{-2mu}\left(#1,#2\right)}
\newcommand{\pg}[2]{\textnormal{PG}\!\left(#1,#2\right)}
\newcommand{\pgTitle}[2]{\textnormal{\textbf{PG}}\!\left(#1,#2\right)}
\newcommand{\ag}[2]{\textnormal{AG}\!\left(#1,#2\right)}
\newcommand{\spread}{\mathcal{D}}
\newcommand{\linset}{\mathcal{L}}
\newcommand{\pointset}{\mathcal{P}}
\newcommand{\mS}{\mathcal{S}}
\newcommand{\qbin}[2]{\genfrac{[}{]}{0pt}{}{#1}{#2}_q}

\renewcommand{\geq}{\geqslant}
\renewcommand{\leq}{\leqslant}

\begin{document}

\maketitle

\begin{abstract} 

In this paper, we investigate the André/Bruck-Bose representation of certain $\FF_q$-linear sets contained in a line of $\pg{2}{q^t}$. We show that {\em scattered $\FF_q$-linear sets of rank $3$} in $\pg{1}{q^3}$ correspond to particular hyperbolic quadrics and that {\em $\FF_q$-linear clubs} in $\pg{1}{q^t}$ are linked to subspaces of a certain $2$-design based on normal rational curves; this design extends the notion of a {\em circumscribed bundle of conics}. Finally, we use these results to construct optimal higgledy-piggledy sets of planes in $\pg{5}{q}$.

\end{abstract}



\section{Introduction}
\subsection{Motivation and overview}

{\em Linear sets} are particular point sets in a finite projective space. They are of interest in finite geometry, and have been studied in recent years through their connections with other topics such as {\em blocking sets}, and their applications in coding theory (see e.g.\ \cite{olga,LavrauwVandeVoordeFieldRed,olgaZullo}). Linear sets generalise the concept of a subgeometry as it has been shown that every linear set is either a subgeometry or the projection of a subgeometry \cite{lunardon}.

The {\em Andr\'e/Bruck-Bose representation} is a way to represent the projective plane over the field $\FF_{q^t}$ with $q^t$ elements, as an incidence structure defined over the subfield $\FF_q$. It is a natural question to study the ABB-representation of certain `nice' sets in the plane, and this has previously been done for sets such as sublines and subplanes \cite{RotteySheekeyVandeVoorde}, (sub)conics \cite{ABBconics} and Hermitian unitals \cite{ABB1}. As such, one can ask the same question about the ABB-representation of $\FF_q$-linear sets; we will give a partial answer in this paper.

We will see that the ABB-representation of a certain type of linear set gives rise to an interesting point set which can be described by using a subspace of a {\em design} of certain {\em normal rational curves}. This design is a generalisation of a well-known design based on the conics of a {\em circumscribed bundle of conics} \cite{BakerBrownEbertFisher}. 

After having introduced the necessary background and definitions in Section \ref{sec1}, we will show in Section \ref{sec2} how to construct this design in a geometric way, and use coordinates to show that the obtained design is, in fact, isomorphic to the design of points and lines in a projective space.

In Sections \ref{sec3} and \ref{sec4}, we will turn our attention towards the ABB-representation of clubs of rank $k$ in $\pg{1}{q^t}$ (Theorem \ref{thmclub}) and scattered linear sets of rank $3$ in $\pg{1}{q^3}$ (Theorem \ref{thmscattered}), both tangent to the line at infinity $\ell_\infty$. 

In Section \textcolor{cyan}{\ref{sec5}}, we first provide the necessary background on higgledy-piggledy sets, and then use the results of Sections \ref{sec3} and \ref{sec4} to show the existence and give explicit constructions of sets of seven planes in $\pg{5}{q}$ in higgledy-piggledy arrangment. This answers an open question of \cite{Denaux}. It was this link which provided the incentive to consider the problem of determining the ABB-representation of linear sets in $\pg{1}{q^3}$.

\subsection{Preliminaries}\label{sec1}
The topics introduced in the following subsections are interrelated; for more information, we refer to \cite{LavrauwVandeVoordeFieldRed}, \cite{RotteySheekeyVandeVoorde} and \cite{IndicatorSet}, respectively.

\subsubsection{Field reduction and Desarguesian spreads}
It is well-known that the vector space $\V{r}{q^t}$ is isomorphic to $\V{rt}{q}$; this isomorphism translates to a correspondence between the associated projective spaces $\pg{r-1}{q^t}$ and $\pg{rt-1}{q}$. Every point of $\pg{r-1}{q^t}$ corresponds to a $1$-dimensional vector space over $\FF_{q^t}$, which is a $t$-dimensional vector space over $\FF_q$, and hence, corresponds to a $(t-1)$-dimensional subspace of $\pg{rt-1}{q}$. In this way, the point set of $\pg{r-1}{q^t}$ gives rise to a set $\spread$ of $(t-1)$-dimensional subspaces of $\pg{rt-1}{q}$ partitioning the point set of $\pg{rt-1}{q}$, that is, they form a {\em $(t-1)$-spread} of $\pg{rt-1}{q}$. Any spread isomorphic to $\spread$ is called a \emph{Desarguesian $(t-1)$-spread}. Similarly, a $(k-1)$-dimensional subspace of $\pg{r-1}{q^t}$ corresponds to a $(kt-1)$-dimensional subspace of $\pg{rt-1}{q}$, spanned by elements of $\spread$.
More formally, we can define the field reduction map $\mathcal{F}_{q,r,t}$ which maps a $(k-1)$-dimensional subspace of $\pg{r-1}{q^t}$ to its associated $(kt-1)$-dimensional subspace of $\pg{rt-1}{q}$. We will omit the subscript of $\mathcal{F}_{q,r,t}$ if the field size and dimensions are clear. If $\mathcal{S}$ is a point set,  we use $\mathcal{F}(\mathcal{S})$ to denote the union of the images of the points in $\mathcal{S}$ under $\mathcal{F}$.

   \subsubsection{The Andr\'e/Bruck-Bose representation}\label{abbintro}
Andr\'e \cite{Andre} and Bruck and Bose \cite{BruckBose} independently derived a representation of a projective plane of order $q^t$ in the projective space $\pg{2t}{q}$.
We refer to this correspondence as the \emph{André/Bruck-Bose representation} or the \emph{ABB-representation}.

Let $H_\infty$ be a hyperplane in $\pg{2t}{q}$ and let $\spread$ be a $(t-1)$-spread in $H_\infty$.
Let $\pointset$ be the set of {\em affine} points (i.e. those of $\pg{2t}{q}$, not contained in $H_\infty$), together with the $q^t+1$ spread elements of $\spread$.
Let $\linset$ be the set of $t$-spaces in $\pg{2t}{q}$ meeting $H_\infty$ in an element of $\spread$, together with the hyperplane at infinity $H_\infty$.
The incidence structure $(\pointset,\linset,I)$, with $I$ the natural incidence relation, is isomorphic to a projective plane of order $q^t$, which is called the \emph{Andr\'e/Bruck-Bose plane} corresponding to the spread $\spread$.
The Andr\'e/Bruck-Bose plane corresponding to a spread $\spread$ is Desarguesian if and only if the spread $\spread$ is Desarguesian.

Now consider $\pg{2}{q^t}$ and let $\ell_\infty$ be a designated line at infinity. Let $H_\infty=\mathcal{F}\left(\ell_\infty\right)$ be a $(2t-1)$-dimensional subspace of $\pg{3t-1}{q}=\mathcal{F}(\pg{2}{q^t})$. Fix a $2t$-space $\mu$ through $H_\infty$. 
It is not hard to see that the Andr\'e/Bruck-Bose representation of an affine point $P$ of $\pg{2}{q^t}$ in $\mu\cong\pg{2t}{q}$ is the point $\mathcal{F}(P)\cap \mu$. We let $\phi$ denote the Andr\'e/Bruck-Bose map on affine points:
$$\phi(P):=\mathcal{F}(P)\cap \mu.$$ The ABB-representation of a point $Q\in\ell_\infty$ is the $(t-1)$-space $\mathcal{F}(Q)$.

\subsubsection{Indicator spaces and Desarguesian subspreads}\label{subsubind}
Finally, we recall the construction of a spread as introduced by Segre \cite{Segre64}. Embed $\Lambda \simeq \pg{rt-1}{q}$ as a subgeometry of $\Lambda^*\simeq \pg{rt-1}{q^t}$. The subgroup of $\mathrm{P}\Gamma\mathrm{L}(rt,q^t)$ fixing $\Lambda$ pointwise is isomorphic to $\mathrm{Aut}(\FF_{q^t}/\FF_q)$. Consider a generator $g$ of this group. One can prove that that there exists an $(r - 1)$-space $\nu$ skew to the subgeometry $\Lambda$ and that a subspace of $\pg{rt-1}{q^t}$ of dimension $s$ is fixed by $g$  if and only if it intersects the subgeometry $\Lambda$ in a subspace of dimension $s$  (see \cite{IndicatorSet}). Let $P$ be a point of $\nu$ and let $L(P)$ denote the $(t-1)$-dimensional subspace generated by the {\em conjugates} of $P$, i.e., $L(P) = \langle P,P^g,\ldots,P^{g^{t-1}}\rangle$. Then $L(P)$ is fixed by $g$ and hence it intersects $\pg{rt-1}{q}$ in a $(t-1)$-dimensional subspace. Repeating this for every point of $\nu$, one obtains a set $\spread$ of $(t-1)$-spaces of the subgeometry $\Gamma$ forming a spread. This spread $\spread$ can be shown to be a Desarguesian spread and $\{\nu,\nu^{g},\ldots,\nu^{g^{t-1}}\}$ is called the {\it indicator set} of $\spread$. An indicator set is also called a set of {\em director spaces} \cite{Segre64}.
It is known from \cite[Theorem 6.1]{IndicatorSet} that for any Desarguesian $(t-1)$-spread of $\pg{rt-1}{q}$ there exist a unique indicator set in $\pg{rt-1}{q^t}$. 

In this paper, {\color{black} we will make use of a particular coordinate system describing a subgeometry $\pi \simeq \pg{t-1}{q}$ in $\pg{t-1}{q^t}$, and for each $s|t$, we will define an $(s-1)$-spread denoted by $\spread_s$ of $\pi$. In the case that $s=t$, this `spread' of $\pi$ is the subspace $\pi$ itself. To describe the set-up,}


let $\sigma$ denote the collineation of $\pg{t-1}{q^t}$ which maps a point with homogeneous coordinates $(x_0,x_1,x_2,\ldots,x_{t-1})$, $x_i\in \FF_{q^t}$, not all zero, onto the point with homogeneous coordinates $(x_{t-1}^q,x_0^q,x_1^q,\ldots,\ldots,x_{t-2}^q)$. The fixed points of $\sigma$ then form a 
subgeometry $\pi\simeq \pg{t-1}{q}$, consisting of all points with homogeneous coordinates $(x,x^q,x^{q^2},\ldots,x^{q^{t-1}})$ for $x\in \FF_{q^t}$. Let $R$ denote the point with coordinates $(1,0,\ldots,0)$, then we see that $R^\sigma=(0,1,\ldots,0)$, $R^{\sigma^2}=(0,0,1,\ldots,0)$ $\ldots$, $R^{\sigma^{t-1}}=(0,0,\ldots,1)$.
Given $R$, every positive divisor $s$ of $t$ induces a unique Desarguesian $(s-1)$-spread $\spread_s$ of $\pi$: consider $\Lambda_s=\mathrm{Fix}(\sigma^s)\simeq \pg{t-1}{q^s}$ and let $\Pi=\langle R,R^{\sigma^s},R^{\sigma^{2s}},\ldots,R^{\sigma^{t-s}}\rangle \cap \Lambda_s$. Then $\{\Pi,\Pi^\sigma,\ldots,\Pi^{\sigma^{s-1}}\}$ is a set of director spaces for $\spread_s$ in $\pg{t-1}{q}$. 

We denote the extension of an element $D$ of $\spread_s$ to $\pg{t-1}{q^t}$ by $\overline{D}$. 

For ease of notation in the case $s=t$, we define the `spread' $\spread_t$ to be equal to $\pi$ and the indicator set of $\pi$ to be the point set $\{R,R^\sigma,\ldots,R^{\sigma^{t-1}}\}$.

\begin{df} Let $$P_x:=\left(\frac{1}{x},\frac{1}{x^q},\frac{1}{x^{q^2}},\ldots,\frac{1}{x^{q^{t-1}}}\right)$$ denote the point of $\pi\simeq \pg{t-1}{q}$ corresponding to $\frac{1}{x}\in \FF_{q^t}^\ast$. 
\end{df}

Note that $P_x=P_y$ if and only if $x/y\in \FF_q$. Furthermore, it is easy to see that $P_x$ is contained in the element $D$ of $\spread_s$ spanned by the points $X,X^\sigma,\ldots,X^{\sigma^{s-1}}$ where $X$ is stabilised by $\sigma^s$ and given by $X=\left(\frac{1}{x},0,\ldots,\frac{1}{x^{q^s}},0,\ldots,\frac{1}{x^{q^{2s}}},0,\ldots,\frac{1}{x^{q^{t-s}}},0,\ldots,0\right)$. Geometrically, the point $X$ is the intersection point of $\overline{D}$ with $\Pi$, where the latter is the director space defining the spread $\spread_s$. It now easily follows that two different points $P_x$ and $P_y$ lie in the same element of $\spread_s$ if and only if $x/y\in \FF_{q^s}$.

\subsubsection{Arcs and normal rational curves}
For any $m\in\NN$ and $k\geq1$, an \emph{$m$-arc} of $\pg{k}{q}$ is a set of $m$ points \emph{in general position}, i.e.\ every $k+1$ points of this point set span $\pg{k}{q}$.

\begin{df}
    Let $1\leq k\leq q$.
    A \emph{normal rational curve} in $\pg{k}{q}$ is a $(q+1)$-arc projectively equivalent to the $(q+1)$-arc corresponding to the coordinates
    \[
        \set{(0,0,\dots,0,1)}\cup\sett{(1,t,t^2,t^3,\dots,t^k)}{t\in\FF_q}\textnormal{.}
    \]
    A point set $\mathcal{C}$ of $\pg{n}{q}$ is a normal rational curve \emph{of degree $k$} if and only if it is a normal rational curve in a $k$-dimensional subspace of $\pg{n}{q}$.
    Note that a normal rational curve of degree $1$ is a line, while one of degree $2$ is a non-degenerate conic.
\end{df}

\begin{res} [{\cite[Theorem 1.18]{Harris}}] \label{uniqueNRC} Consider a $(k+2)$-arc $\mathcal{A}$ in $\pg{k-1}{q}$, $k+1\leq q$, then there exists a unique normal rational curve of degree $k-1$ through all points of $\mathcal{A}$.
\end{res}

\begin{res} [{\cite[Lemma 27.5.2(i)]{Hirschfeldgalois}}] \label{projNRC} Let $\mathcal{C}$ be a normal rational curve \emph{of degree $k-1$} in $\pg{k-1}{q}$, and let $P\in \mathcal{C}$. The projection of $\mathcal{C}\setminus\{P\}$ from $P$ onto a $(k-2)$-space disjoint from $P$ is a point set of size $q$ contained in a normal rational curve of degree $k-2$. If $k+1\leq q$, then this normal rational curve is unique.
\end{res}

\subsubsection{The ABB-representation of sublines and subplanes}

The ABB-represention of $\FF_{q^k}$-sublines and tangent subplanes of $\pg{2}{q^t}$ was studied in \cite{RotteySheekeyVandeVoorde}.

In this paper, we will make use of the following cases tackled there:

\begin{res}[\cite{RotteySheekeyVandeVoorde}]\label{Res_SublinesTangent}
    \begin{itemize}
        \item[(a)] The affine points of an $\FF_{q}$-subline in $\pg{2}{q^t}$ tangent to $\ell_\infty$ correspond to the points of an affine line in the ABB-representation and vice versa.
        \item[(b)] Suppose that $q\geq t$ and $k\mid t$. Let $m$ be an $\FF_q$-subline of $\pg{2}{q^t}$ external to $\ell_\infty$ where the smallest subline containing $m$ and tangent to $\ell_\infty$ is an $\FF_{q^k}$-subline. Then the ABB-representation of $m$ is a set of points $\mathcal{C}$ in $\pg{2t}{q}$ such that     \begin{enumerate}
        \item $\mathcal{C}$ is a normal rational curve of degree $k$ contained in a $k$-space intersecting $H_\infty$ in an element of $\spread_k$.
        \item its $\FF_{q^t}$-extension $\mathcal{C}^*$ to $\pg{2t}{q^t}$ intersects the indicator set $\set{\Pi,\Pi^\sigma,\dots,\Pi^{\sigma^{k-1}}}$ of $\spread_k$ in $k$ conjugate points.
    \end{enumerate} and vice versa, any set $\mathcal{C}$ with those properties gives rise to the point set of an $\FF_q$-subline, external to $\ell_\infty$.
        
    \end{itemize}
\end{res}

\subsubsection{Linear sets}
For a more thorough introduction to linear sets, we refer to \cite{LavrauwVandeVoordeFieldRed,olga}.
In this paper, we will only be concerned with linear sets on a projective line, and we will use the geometrical point of view on linear sets using Desarguesian spreads. Let $\spread$ be the Desarguesian spread in $\pg{2t-1}{q}$ obtained as the image of the field reduction map on points of $\pg{1}{q^t}$. Then a set $\mathcal{S}$ in $\pg{1}{q^t}$ is an $\FF_q$-linear set of rank $k$ if and only if there is a $(k-1)$-dimensional subspace $\pi$ of $\pg{2t-1}{q}$ such that $$\mathcal{F}(\mathcal{S})=\mathcal{B}(\pi),$$ where $\mathcal{B}(\pi)$ is the set of elements of $\spread$ meeting $\pi$ in at least a point.

\begin{df} We denote the $\FF_q$-linear set $\mathcal{S}$ such that $\mathcal{F}(\mathcal{S})=\mathcal{B}(\pi)$ by $L_{\pi}$.
\end{df}

The {\em weight} of a point $P$ in $L_\pi$ is $w+1$ if $w$ is the dimension of $\mathcal{F}(P)\cap\pi$. Note that the weight of a point in a linear set is only well-defined if we specify the subspace $\pi$ defining $L_\pi$.

In this article, we focus on {\em scattered $\FF_q$-linear sets} in $\pg{1}{q^3}$ and {\em clubs} in $\pg{1}{q^t}$. A scattered linear set of rank $k$ in $\pg{1}{q^t}$ is an $\FF_q$-linear set of rank $k$ consisting of $\frac{q^k-1}{q-1}$ points. We see that all the points of a scattered linear set have weight one. If $L_\pi$ is a scattered linear set, then the subspace $\pi$ is called {\em scattered} (with respect to the Desarguesian spread $\spread$).
A {\em $t$-club} of rank $k$ is an $\FF_q$-linear set $L_\pi$ such that there is one point of weight $t$ and all other points have weight one; if $t=k-1$, this set is simply called a {\em club}. The point of weight $t$ is called the {\em head} of the club. As for the weight of the points in the linear set, we see that the head of the club is only well-defined with respect to the subspace $\pi$.

%
%
%

We have the following result about the possible intersection of an $\FF_q$-linear set and an $\FF_q$-subline.
\begin{res}[{\cite[Theorem $8$]{LavrauwVandeVoordeLinSetLine}}]\label{Res_LinearSetIntersectionSubline}
    An $\FF_q$-subline intersects an $\FF_q$-linear set of rank $k$ of $\pg{1}{q^t}$ in at most $k$ or precisely $q+1$ points.
\end{res}

The following results on clubs and scattered linear sets on a projective line reveal some useful geometric properties. Note that the authors of \cite{LavrauwVandeVoordeLinSetLine} did not include the necessary condition that $q\geq 3$.

\begin{res}[{\cite[Corollary $13$ and $15$]{LavrauwVandeVoordeLinSetLine},\cite[Theorem 3.7.4]{mijnthesis}}]\label{Res_LinearSetClub}
    Suppose that $q\geq 3$.   \begin{itemize}\item[(a)] If $\mS$ is a club of $\pg{1}{q^t}$, $\mS\not\simeq\pg{1}{q^2}$, then through two distinct non-head points of $\mS$, there exists exactly one $\FF_q$-subline contained in $\mS$, which necessarily contains the head of the club.
    \item[(b)]  If $\mS$ is a scattered linear set of rank $3$ of $\pg{1}{q^3}$, then through two distinct points of $\mS$, there are exactly two $\FF_q$-sublines contained in $S$.
 \item[(c)] Let $q\geq 5$. Consider a scattered plane $\pi$ with respect to the Desarguesian $2$-spread $\spread$ in $\pg{5}{q}$ and let $r\in \pi$. Then there is exactly one plane $\pi'\neq \pi$ through $r$ such that $\mathcal{B}(\pi)=\mathcal{B}(\pi')$.
    \end{itemize}
\end{res}


\section{Generalising the circumscribed bundle of conics}\label{sec2}

In order to characterise the ABB-representation of clubs, tangent to $\ell_\infty$, we will introduce a block design $\mathcal{H}$ embedded in $\pg{t-1}{q}$, where blocks are certain normal rational curves. In the particular case when $t=3$, this design is known as the design arising from a {\em circumscribed} bundle of conics. In \cite{BakerBrownEbertFisher}, the authors describe three types of {\em projective bundles}, which they define to be a collection of $q^2+q+1$ conics mutually intersecting in exactly one point. The circumscribed bundles are {\em bundles} in the classical algebraic sense: given three conics in the bundle defined by equations $f=0$, $g=0$, $h=0$ where $h$ is not an $\FF_q$-linear combination of $f$ and $g$, every conic in the bundle is defined by $\lambda f+\mu g+\nu h=0$ for some $\lambda,\mu,\nu \in \FF_q$.

We see that the design $(\mathcal{P},\mathcal{B})$ where points $\mathcal{P}$ are the points of $\pg{2}{q}$, blocks $\mathcal{B}$ are the conics of the projective bundle, and incidence is inherited, forms a projective plane. The {\em circumscribed} bundle consists of all conics in $\pg{2}{q}$ whose extension to $\pg{2}{q^3}$ contains three fixed conjugate points $R,R^q,R^{q^2}$ spanning $\pg{2}{q^3}$. It can be deduced from \cite{LavrauwVandeVoordeLinSetLine} that the projective plane constructed via the circumscribed bundle is the Desarguesian plane $\pg{2}{q}$. The design here will be a natural generalisation of this construction; for $t$ prime, its definition is straightforward but for $t$ non-prime, extra care must be taken.

Let $e_0,e_1,\ldots,e_{t-1}$ be the standard basis vectors of length $t$ (with $1$ in the $(i+1)$-th position and zero elsewhere) and let $\langle v\rangle$ denote the projective point of $\pg{t-1}{q^t}$ with homogeneous coordinates given by $v$.

\begin{lm} \label{NRCs}  (Using the notations introduced in \ref{subsubind}) Consider the points $R^{\sigma^i}=\langle e_i\rangle$, $i=0,\ldots,t-1$, in $\pg{t-1}{q^t}$ and two points $P_a\neq P_b$ in $\pi \simeq \pg{t-1}{q}$. Let $s$ be the smallest integer such that $a/b\in \FF_{q^s}$ and let $D$ be the element of the Desarguesian $(s-1)$-spread  $\spread_{s}$ containing $P_a$ and $P_b$. Then 
\begin{enumerate}
\item there is a unique normal rational curve $\mathcal{C}^{a,b}$ of degree $s-1$ through $P_a$ and $P_b$, contained in $\overline{D}$, and meeting the indicator spaces $\{\Pi,\Pi^\sigma,\ldots,\Pi^{\sigma^{s-1}}\}$ in $s$ conjugate points. 
\item the points of $\mathcal{C}^{a,b}$ are given by $\{K^{a,b}_{u,v}|u,v\in \FF_{q^t}\}$ where 
$$K^{a,b}_{u,v}:=\left\langle \sum_{i=0}^{s-1}\prod_{j=0,j\neq i}^{s-1} (a^{q^j}u-b^{q^j}v)w_i\right\rangle;$$
and the conjugate points are $Q,Q^\sigma,\ldots,Q^{\sigma^{s-1}}$ where $Q^{\sigma^{i-1}}=\langle w_i \rangle$ with

\begin{align} w_0&=a(\frac{1}{a},0,\ldots,0,\frac{1}{a^{q^s}},0,\ldots,0,\frac{1}{a^{q^{2s}}},\ldots,\frac{1}{a^{q^{t-s}}},0,\ldots,0)\nonumber\\
w_1&=a^q(0,\frac{1}{a^q},\ldots,0,\frac{1}{a^{q^{s+1}}},0,\ldots,0,\frac{1}{a^{q^{2s+1}}},\ldots,\frac{1}{a^{q^{t-s+1}}},0,\ldots,0)\nonumber\\
&\;\;\vdots\nonumber\\
 w_{s-1}&=a^{q^{s-1}}(0,\ldots,\frac{1}{a^{q^{s-1}}},0,\ldots,0,\frac{1}{a^{q^{t-1}}}).\label{conj}\end{align}

\item $\mathcal{C}^{a,b}$ meets $\pi$ in $q+1$ points, determined by the points $P_{au-bv}$ where $u,v\in \FF_q$.
\end{enumerate}
\end{lm}

\begin{proof} Recall that, given $D$, the set of $s$ conjugate points contained in both the indicator spaces and in $\overline{D}$ is fixed. As discussed in Section \ref{subsubind}, it is easy to check that the coordinates corresponding to this set $\{Q,Q^\sigma,\ldots,Q^{\sigma^{s-1}}\}$ of conjugate points is given by the vectors in \eqref{conj}. By Result \ref{uniqueNRC}, we know that there is a unique normal rational curve of degree $s-1$ containing the $s$ conjugate points and the points $P_a$ and $P_b$.

It is well-known (see e.g.\ \cite[Example 1.17]{Harris}) that $\mathcal{C}^{a,b}$ as given in the statement of the lemma defines a normal rational curve; the degree of this curve is $d$ if the point set $\{(a^{q^i},b^{q^i})|i=0,\ldots,t-1\}$ in $\pg{1}{q^t}$ consists of $d+1$ different points. Recall that $s$ is the smallest integer such that $a/b\in \FF_{q^s}$, and hence, $s$ is the smallest integer for which $(\frac{a}{b})^{q^s}=\frac{a}{b}$. This means that the point set $\{(a^{q^i},b^{q^i})|i=0,\ldots,t-1\}$ consists of $s$ different points, implying that the degree of $\mathcal{C}^{a,b}$ is indeed $s-1$.

Now consider the point $K^{a,b}_{0,1}=\langle (-1)^{s-1}\sum_{i=0}^{s-1}(\prod_{j=0,j\neq  i}^{s-1}b^{q^j})w_i\rangle$. By dividing by $(-1)^{s-1}\prod_{j=0}^{s-1}b^{q^j}$, we find that this point has coordinates $(\frac{1}{b},\frac{1}{b^q},\ldots,\frac{1}{b^{q^{t-1}}})$, and hence, is the point $P_b$. Similarly, $K^{a,b}_{1,0}$ is the point $P_a$, and we see that $\mathcal{C}^{a,b}$ indeed passes through $P_a$ and $P_b$.

Note that $K_{b^{q^{i'}},a^{q^{i'}}}^{a,b}=\langle w_{i'}\rangle$, $i'=0,1,\dots,s-1$. In other words, $\mathcal{C}^{a,b}$ indeed contains the $s$ conjugate points $Q,Q^\sigma,\ldots,Q^{\sigma^{s-1}}$.

{\color{black} 
Finally, if $u,v\in \FF_q$, and using that $b/a\in \FF_{q^s}$, it can be checked that $P_{au-bv}=K_{u,v}^{a,b}$, and vice versa, if a point $K_{u,v}^{a,b}$ lies in $\pi$, then it follows that $u,v\in \FF_q$. This means that the $q+1$ different points of the form $P_{au-bv}$, where $u,v\in \FF_q$, are precisely those in $\mathcal{C}^{a,b}\cap\pi$; the normal rational curve $\mathcal{C}^{a,b}$ meets $\pi$ in a normal rational curve of $\pi$.}

\end{proof}

\begin{rmk} The fact that $P_{au-bv}$ defines a normal rational curve in the subgeometry $\pi$ as seen in Lemma \ref{NRCs} also follows by considering the cyclic model of $\pg{t-1}{q}$ (see e.g.\ \cite{cyclicmodel}): it is well-known that the inverse of a line in this model is a normal rational curve. In Lemma \ref{NRCs}, we have described the extension of this normal rational curve to $\pg{t-1}{q^t}$.
\end{rmk}

\begin{df} Consider a subgeometry $\pi$ $\simeq\pg{t-1}{q}$ arising as the set of fixed points of a collineation $\sigma$ of $\pg{t-1}{q^t}$, and let $R$ be a point such that the points $R,R^\sigma,R^{\sigma^2},\ldots,R^{\sigma^{t-1}}$ span $\pg{t-1}{q^t}$. Consider the Desarguesian subspreads $\spread_s$ for every $1<s\leq t$, $s|t$, as defined in Subsection \ref{subsubind}. Let $\mathcal{H}$ denote the following incidence structure:
\begin{itemize}
\item Points $\mathcal{P}$ are the points of $\pi$;
\item Let $P$ and $Q$ be two distinct points of $\pi$, and $s$ be the smallest integer such that $P,Q$ are contained in the same element of $\spread_s$, say $D$. Then the unique block through $P$ and $Q$ is the set of points of $\pi$ contained in the normal rational curve of degree $s-1$ through $P,Q$ and the intersection points of $\overline{D}$ with the indicator spaces $\Pi,\Pi^\sigma,\ldots,\Pi^{\sigma^{s-1}}$.
\end{itemize}
\end{df}
In the case $t=3$, the above construction reproduces the design obtained from the circumscribed bundle of conics; we have $q^2+q+1$ points in $\mathcal{H}$. Since $t$ is prime, necessarily $s=3$ for all pairs of points. Recall that a normal rational curve of degree $2$ is a conic, and hence, the block through two points $P$ and $Q$ is simply the intersection of $\pg{2}{q}$ with the unique conic through $P,Q,R,R^\sigma$ and $R^{\sigma^2}$. We see that indeed, these five points are in general position, and that the unique conic through these $5$ points intersects $\pi$ in a subconic.

In the following Lemma, we will use the axiom of Veblen-Young to deduce that the point-line incidence geometry $\mathcal{H}$ is isomorphic to the point-line incidence geometry of a projective space, which is necessarily $\pg{t-1}{q}$. Note that this approach does not reprove the case $t=3$. 

\begin{thm} \label{Hisprojective} Let $t>3$. The incidence structure $\mathcal{H}$ is a $2$-$(\theta_{t-1},q+1,1)$ design, isomorphic to the design of points and lines in $\pg{t-1}{q}$
\end{thm}
\begin{proof}  The fact that $\mathcal{H}$ determines a $2$-$(\theta_{t-1},q+1,1)$ design follows directly from Lemma \ref{NRCs} and the fact that there are $\theta_{t-1}$ points in $\pg{t-1}{q}$. In order to show that it is isomorphic to the design of points and lines in $\pg{t-1}{q}$, we will verify that the Veblen-Young axiom holds in $\mathcal{H}$. More precisely, we will show that if the block through two points $A$ and $B$ (denoted by $AB$) has a point in common with the  block $CD$,  then the block $AD$ has a point in common with the block $BC$.

Let $A=P_a$, $B=P_b$, $C=P_c$ and $D=P_d$ be four different points of $\pi$ and assume that there is a point $P$ on $AB$ and $CD$. By Lemma \ref{NRCs}, $P=P_{au_0-bv_0}$ for some $u_0,v_0\in \FF_q$. Similarly, $P=P_{cu_1-dv_1}$ for some $u_1,v_1\in\FF_q$. Since $P=P_{au_0-bv_0}=P_{cu_1-dv_1}$, it follows that $(au_0-bv_0)/(cu_1-dv_1)\in \FF_q$, so there exists an element $\lambda\in \FF_q$ with
$$au_0-bv_0=\lambda(cu_1-dv_1),$$
or equivalently,
$$au_0+\lambda dv_1=bv_0+\lambda cu_1.$$
This implies that $P_{au_0+\lambda dv_1}=P_{bv_0+\lambda cu_1}$. Since $\lambda, u_0,v_0,u_1,v_1\in \FF_q$, the left hand side is a point of $\mathcal{C}^{a,d}$ in $\pi$, and the right hand side is a point of $\mathcal{C}^{b,c}$ in $\pi$. Hence, the blocks $AD$ and $BC$ have a point in common.
\end{proof}

It follows that $\mathcal{H}$ admits {\em subspaces}, and that we can talk about the dimension of this subspace. To avoid confusing with subspaces of $\pg{n}{q}$, we will denote subspaces of $\mathcal{H}$ by $\mathcal{H}$-subspaces.
These $\mathcal{H}$-subspaces will appear in the characterisation of the ABB-representation of a club, tangent to $\ell_\infty$ and with head different from $P_\infty$.

\section{Tangent clubs of rank \texorpdfstring{$k$}{k} in \texorpdfstring{$\pgTitle{1}{q^t}$}{PG(1,qt)}} \label{sec3}


As in Subsection \ref{abbintro}, we let $\ell_\infty$ be the line of $\pg{2}{q^t}$ such that the ABB-representation of $\pg{2}{q^t}$ has $H_\infty=\mathcal{F}(\ell_\infty)$ as the hyperplane at infinity of $\mu=\pg{2t}{q}$. In this section, we will consider the ABB-representation of a linear set contained in a line $\ell\neq \ell_\infty$ of $\pg{2}{q^t}$. We will denote $P_\infty=\ell\cap \ell_\infty$ and the corresponding spread element by $\pi_\infty=\mathcal{F}(P_\infty)$. Let $\Pi$ be the $t$-space in $\pg{2t}{q}$  through $\pi_\infty$ containing all the points of $\phi(\ell\setminus\{P_\infty\})$.

\begin{rmk} The different perspectives on linear sets lead to different possible approaches for studying their ABB-representation. The (affine part of) the ABB-representation of a linear set $L_\pi$ on a projective line $\pg{1}{q^t}$ can be seen as the intersection of the set $\mathcal{B}(\pi)$ with a $t$-dimensional subspace containing a fixed spread element of $\spread$. Furthermore, since a linear set of rank $3$ can be seen as the projection of a subplane, and the ABB-representation of tangent and secant subplanes is understood (see \cite{RotteySheekeyVandeVoorde}), in Theorem \ref{thmscattered} we are looking to characterise the projection of certain normal rational scrolls.
The two above approaches make it possible to give a description of the ABB-representation of a linear set; for example, the ABB-representation of a scattered linear set of rank 3 tangent to the line at infinity is the projection of a normal rational scroll. However, we found these descriptions insufficient to be able to fully characterise the ABB-representation of the linear sets as done with the approach of our paper.
\end{rmk}

%
%

\subsection{Counting clubs of \texorpdfstring{$\pgTitle{1}{q^t}$}{PG(1,q\textasciicircum t)} }
In order to characterise the ABB-representation of clubs, we will count the number of different clubs with a fixed head. Note that we are not dealing with {\em (in)-equivalence} nor {\em simplicity} here; in general, clubs of rank $t$ in $\pg{1}{q^t}$ are equivalent but the same is not true for clubs of rank $k<t$ (see e.g.\ \cite{equivalence} and \cite{NPSZ}). Furthermore, in general, clubs are not necessarily {\em simple}: if $\mathcal{B}(\pi)=\mathcal{B}(\pi')$ is a club for two subspaces $\pi$ and $\pi'$ sharing a point, then it is not true that necessarily $\pi=\pi'$, nor is the head of the club determined by the point set itself (this was already noted in \cite{clubs}). However, if we specify the head of the club, we can show the following statement:
\begin{lm} \label{heads}Let $L_\pi=L_{\pi'}$ be two clubs of rank $k$ in $\pg{1}{q^t}$ with head $P$ (that is, $\pi$ and $\pi'$ are $(k-1)$-dimensional spaces and $\pi\cap \mathcal{F}(P)$ and $\pi'\cap \mathcal{F}(P)$ are $(k-2)$-dimensional). If there is a point $r$ in $\pi\cap\pi'$, and not in $\mathcal{F}(P)$, then $\pi=\pi'$. Hence, there are $\frac{q^t-1}{q-1}$ subspaces $\pi'$ such that $L_\pi=L_{\pi'}$ is a club with head $P$.
\end{lm}
\begin{proof} Let $\pi$ and $\pi'$ be as in the statement of the lemma and assume that that $\pi\neq\pi'$. Then there exists a point $s\in \pi$, not in $\pi'$, nor in $\mathcal{F}(P)$; since $\mathcal{B}(\pi)=\mathcal{B}(\pi')$, it follows that $\mathcal{B}(s)$ intersects $\pi'$ in a point $s'$. The line through $r$ and $s$ meets $\mathcal{F}(P)$ in a point, as does the line through $r$ and $s'$; hence, both define the unique $\FF_q$-subline through $\mathcal{F}^{-1}(\mathcal{B}(r))$,  $\mathcal{F}^{-1}(\mathcal{B}(s))$ and $P$ in $L_\pi$. But there is a unique transversal line through $r$ to the regulus defined by the elements $\mathcal{B}(r),\mathcal{B}(s), \mathcal{F}(P)$, a contradiction.
Finally, it is well-known that the elementwise stabiliser of the Desarguesian spread $\spread$ acts transitively on the points inside a spread element (see e.g.\ \cite[Lemma 4.3]{LavrauwVandeVoordeFieldRed}). Hence, for all $\frac{q^t-1}{q-1}$ points $u$ in $\mathcal{B}(r)$ we find a unique subspace $\pi''$ through $u$ with $\mathcal{B}(\pi'')=\mathcal{B}(\pi)$ and $\pi''\cap \mathcal{F}(P)$ a $(k-2)$-dimensional space, so the statement follows.
\end{proof}

\subsection{Clubs with head \texorpdfstring{$P_\infty$}{P-infty}}
The characterisation of the ABB-representation of clubs with head $P_\infty$ easily follows by using the different perspectives on linear sets.
\begin{prop}\label{trivial}
    Suppose that $q\geq 3$.
    A point set $\mathcal{S}$ of $\pg{1}{q^t}$ is an $\FF_q$-linear club of rank $k$ with head $P_\infty$ if and only if the ABB-representation of $S\setminus\set{P_\infty}$ is an affine $(k-1)$-space of $\Pi$.
\end{prop}
\begin{proof}
Let $M$ be an affine point set contained in the line $\ell\neq \ell_\infty$ of $\pg{2}{q^t}$. Recall that the ABB-representation of $M$ can be obtained from intersecting the image of $M$ under the field reduction map with the subspace $\mu$ of dimension $2t$ through $H_\infty$, where $H_\infty$ is the $(2t-1)$-dimensional space $\mathcal{F}(\ell_\infty)$. We denote the subspace $\mathcal{F}(\ell)\cap \mu$ containing the ABB-representation of the affine points of $\ell$ by $\Pi$. The ABB-representation of $M$ is the intersection of spread elements $\mathcal{F}(P)$, where $P\in M$, with $\Pi$. We claim that if $M$ is the affine point set of a club with head $P_\infty$, the points of this intersection form a subspace and vice versa.

First note that if $\nu$ is an affine $(k-1)$-space of $\Pi$, and $\bar{\nu}$ denotes its projective completion, trivially, $\mathcal{B}(\bar{\nu})$ is the set of elements of the Desarguesian spread meeting a $(k-1)$-space and intersecting $P_\infty$ in a $(k-2)$-space; that is, it defines a club of rank $k$ with head $P_\infty$. 

Vice versa, suppose that $M$ is the affine point set of a club with head $P_\infty=\ell\cap \ell_\infty$. By definition, there is a $(k-1)$-dimensional subspace $\pi$ contained in $\mathcal{F}(\ell)$ such that $\mathcal{S}=\mathcal{B}(\pi)$, and furthermore, such that $\pi$ meets $H_\infty$ in a $(k-2)$-dimensional space. If $\pi$ is a subspace of $\Pi$, then we are done. Otherwise, let $v$ be a point of $\Pi$ lying in a spread element of $\mathcal{B}(\pi)$, different from $\mathcal{F}(P_\infty)=\pi_\infty$, then by Lemma \ref{heads}, there is a subspace $\pi'$ through $v$ such $\mathcal{B}(\pi')=\mathcal{B}(\pi)$. Since $\pi'$ lies in $\Pi$, we find that $\pi'$ is the intersection of $\mathcal{B}(\pi)$ with $\Pi$ and the statement follows.
\end{proof} 


Let $\qbin{n}{k}$ denote the number of $(k-1)$-dimensional subspaces of $\pg{n-1}{q}$, that is, $$\qbin{n}{k}=\frac{(q^n-1)(q^{n-1}-1)\cdots (q-1)}{(q^k-1)(q^{k-1}-1)\cdots(q-1)},$$
and let $\theta_m$ be the number of points in $\pg{m-1}{q}$, that is,
\[
    \theta_m=\frac{q^m-1}{q-1}.
\]

\begin{prop}\label{heads2} There are 
$q^{t-k+1}\qbin{t}{k-1}$
clubs $L_\pi$ of rank $k$ with head $P_\infty$. 
\end{prop}
\begin{proof} There are $\qbin{t}{k-1}$ {\color{black} subspaces of dimension $k-2$} in $\pi_\infty=\mathcal{F}(P_\infty)$, and each of them lies on  $\frac{q^{2t-k+1}-1}{q-1}-\frac{q^{t-k+1}-1}{q-1}$ subspaces of dimension $k-1$, not contained in $\pi_\infty$. By Lemma \ref{heads}, there are $\theta_{t-1}$ of such $(k-1)$-spaces $\pi$ giving rise to the same club. Hence, we find that there are
$$\frac{\qbin{t}{k-1}(\frac{q^{2t-k+1}-1}{q-1}-\frac{q^{t-k+1}-1}{q-1} )}{\frac{q^t-1}{q-1}} =q^{t-k+1}\qbin{t}{k-1}$$ clubs with head $P_\infty$.

\end{proof}
\subsection{Clubs with head different from \texorpdfstring{$P_\infty$}{P-infty} }
\begin{prop}\label{Prop_NumberClubsHeadNotInftyevenmoregeneral} Let $H$ and $P_\infty$ be two different points of $\pg{1}{q^t}$. Then there exist $\qbin{t}{k-1}$ clubs $L_\pi$ through $P_\infty$ with head $H$, {where $\pi$ is a $(k-1)$-space}.
    
     Furthermore, there are $q^t\qbin{t}{k-1}$ clubs $L_\pi$, where $\pi$ is a $(k-1)$-space, containing $P_\infty$, with head different from $P_\infty$.
\end{prop}
\begin{proof}
    Let $\gamma:=\mathcal{F}(H)$.
    A $(k-2)$-space $g$ in $\gamma$ and a point $P$ in $\pi_\infty$ span a $(k-1)$-space $\langle g,P\rangle$ which defines a club with head $H$ and containing $P_\infty$.
   By Lemma \ref{heads}, every club with head $H$ and containing $P_\infty$ is defined by exactly $\theta_{t-1}$ such $(k-1)$-spaces, so the total 
 number of clubs through a fixed head point $H\neq P_\infty$ and containing $P_\infty$ is
    \[
        \frac{\qbin{t}{k-1}\theta_{t-1}}{\theta_{t-1}}\textnormal{.}
    \]
    There are $q^t$ choices for a point $H\neq P_\infty$, and each subspace $\pi$ defines a unique $H$, so there are $q^t\qbin{t}{k-1}$ clubs $L_\pi$, where $\pi$ is a $(k-1)$-space  and the head is different from $P_\infty$.
\end{proof}

\begin{prop}\label{Prop_Numbercones}
    There exists $q^t\qbin{t}{k-1}$ cones in $\Pi$ with vertex a point $H\notin \pi_\infty$ and base a $(k-2)$-dimensional subspace of the $2$-design $\mathcal{H}$.
    
\end{prop}
\begin{proof}
From Theorem \ref{Hisprojective}, it follows that the number of $(k-2)$-dimensional subspaces of $\mathcal{H}$ equals the number of $(k-2)$-spaces in $\pg{t-1}{q}$, that is, $\qbin{t}{k-1}$. Furthermore, there are $q^t$ points in $\Pi$, not in $\pi_\infty$, each of which defines a unique cone with vertex that point and base a $(k-2)$-dimensional subspace of $\mathcal{H}$. 
\end{proof}

In order to characterise the ABB-representation of a club with head, different from the point at infinity, we need the following Lemma from \cite{sam}.
\begin{lm}[{\cite[Lemma 5.7]{sam}}] \label{lemmahyp}  Assume that $\mS$ is a point set in $\pg{n}{q}$, $q\geq 4$, with the property that every line intersects $\mS$ in $0,1,q$ or $q+1$ points. Then there exists a hyperplane $H$ in $\pg{n}{q}$ such that either $\mS\subseteq H$ or $\mS^c\subset H$, where $\mS^c$ denotes the complement of $\mS$ in $\pg{n}{q}$.

\end{lm}
\begin{thm} \label{thmclub}A set $\mathcal{S}$ is an $\FF_q$-linear club of rank $k$ in $\pg{1}{q^t}$ containing $P_\infty$ and with head $H\neq P_\infty$, if and only if $\phi(\mathcal{S}\setminus\{P_\infty\})$, the ABB-representation of $\mathcal{S}\setminus\{P_\infty\}$ in $\pg{2t}{q}$, is the affine point set of a cone with vertex $\phi(H)$ and base an $\mathcal{H}$-subspace of dimension $(k-2)$  in $\mathcal{F}(P_\infty)$ (the spread element corresponding to $P_\infty$).

\end{thm}
\begin{proof} Let $\mathcal{S}$ be an $\FF_q$-linear club of rank $k$ containing $P_\infty$ and with head $H\neq P_\infty$, and let $\phi(H)$ be the ABB-representation of the head $H$. Let $Q\notin\{H,P_\infty\}$ be a point of $\mathcal{S}$. By Result \ref{Res_LinearSetClub}(a), we know that the subline through $H,Q,P_\infty$ is contained in $\mathcal{S}$. By Result \ref{Res_SublinesTangent}(a), the ABB-representation of the points, different from $P_\infty$, of this subline are the affine points of the line through $\phi(H)$ and $\phi(Q)$. In other words, the $q^{k-1}-1$ points of $\mS\setminus\{H,P_\infty\}$ are contained in $\frac{q^{k-1}-1}{q-1}$ lines through $\phi(H)$, that is, they form a cone with vertex $\phi(H)$. The projective completions of those lines meet $\mathcal{F}(P_\infty)$ in a set $\mathcal{K}$ of $\frac{q^{k-1}-1}{q-1}$ points.

Let $R_i$, $i=1,2$, be two different points of $\mathcal{K}$, and let $Q_i$ be a point on the line through $\phi(H)$ and $R_i$, different from $\phi(H)$ and $R_i$. We have that $Q_i=\phi(S_i)$ for some point $S_i\in \mathcal{S}$. Moreover, from Result \ref{Res_LinearSetClub}(a), we know that the subline $m$ through $H,S_1,S_2$ is contained in $\mathcal{S}$. Let $s$ be the integer such that the smallest subline containing $m$ and tangent to $\ell_\infty$ is an $\FF_{q^s}$-subline. Then by Result \ref{Res_SublinesTangent}(b), we know that the affine points of this subline correspond to a normal rational curve $\mathcal{C}$ through $\phi(H),Q_1, Q_2$, contained in an $s$-space meeting $\mathcal{F}(P_\infty)$ in an element $D$ of $\spread_s$, whose $\FF_{q^t}$-extension intersects the indicator set of $\spread_s$ in $s$ conjugate points. Note that $R_1,R_2$ are contained in $D$, and hence, $D$ is the unique element of $\spread_s$ containing $R_1,R_2$.

By Result \ref{projNRC}, the projection of the normal rational curve $\mathcal{C}$ from the point $\phi(H)\in \mathcal{C}$ onto $H_\infty$ is contained in a normal rational curve; this curve is contained in $\pi_\infty$, goes through $R_1$, $R_2$ and the extension contains the same points in $H_\infty$ as $\mathcal{C}$ did. Hence, the block of the design $\mathcal{H}$ through $R_1,R_2$ contains $q$ points of $\mathcal{K}$. It follows that $\mathcal{K}$ is a point set meeting every block in $0,1,q$ (or $q+1$) points. By Theorem \ref{Hisprojective}, $\mathcal{H}$ is isomorphic to the point-line design of  $\pg{t-1}{q}$ so we may use Lemma \ref{lemmahyp} to conclude that $\mathcal{K}$ or its complement must be contained in a hyperplane $\mu$ of the design $\mathcal{H}$. Since $\frac{q^t-1}{q-1}-|\mathcal{K}|>\frac{q^{t-1}-1}{q-1}$, the latter possibility does not occur.
We can repeat the same reasoning in the $(t-2)$-dimensional $\mathcal{H}$-subspace $\mu$: all blocks of $\mu$ meet $\mathcal{K}$ in $0,1,q$ or $q+1$ points, and since $\frac{q^{t-1}-1}{q-1}-|\mathcal{K}|>\frac{q^{t-2}-1}{q-1}$, $\mathcal{K}$ is contained in a hyperplane of $\mu$, that is, a $(t-3)$-dimensional $\mathcal{H}$-subspace. Continuing in this fashion, we conclude that $\mathcal{K}$ is contained in a $(k-2)$-dimensional $\mathcal{H}$-subspace . Since $|\mathcal{K}|=\frac{q^{k-1}-1}{q-1}$, equality holds.

Furthermore, by Propositions \ref{Prop_Numbercones} and \ref{Prop_NumberClubsHeadNotInftyevenmoregeneral}, the number of such cones equals the number of $\FF_q$-linear club of rank $k$ containing $P_\infty$ and with head $H\neq P_\infty$, and the theorem follows.
\end{proof}

\section{Tangent scattered linear sets of rank \texorpdfstring{$3$}{3} in \texorpdfstring{$\pgTitle{1}{q^3}$}{PG(1,q\textasciicircum3)} }\label{sec4}

We continue to use the same notations as in the previous section, as introduced in Subsection \ref{abbintro}.

\begin{prop}\label{combprop}
    Suppose that $q\geq 5$. 
    Let $\mathcal{U}$ be a point set of $\ag{3}{q}$ with the following three properties:
    \begin{enumerate}
        \item for each line $\ell$ holds that $|\ell\cap\mathcal{U}|\in\{0,1,2,q\}$,
        \item through each point of $\mathcal{U}$, there exist precisely two lines that are contained in $\mathcal{U}$, and
        \item $|\mathcal{U}|=q^2+q$.
    \end{enumerate}
    Let $\pi_\infty$ be the plane at infinity when embedding $\ag{3}{q}$ in $\pg{3}{q}$.
    Then $\mathcal{U}$ is the affine part of a hyperbolic quadric in $\pg{3}{q}$ that intersects $\pi_\infty$ in a non-degenerate conic.
\end{prop}
\begin{proof}
    We claim that the intersection of a plane $\sigma$ with $\mathcal{U}$ is either a cap or the union of two distinct lines. 
    First note that it impossible for $\sigma\cap\mathcal{U}$ to contain two lines $\ell_1$, $\ell_2$ and a point $R\in\mathcal{U}\setminus\left(\ell_1\cup\ell_2\right)$: in this case, since $q\geq 5$, we find that there are at least $3$ lines through $R$ meeting $\ell_1$ and $\ell_2$ in distinct points, which forces those lines to be contained in $\mathcal{U}$ by Property 1., contradicting Property 2.

    Suppose that $\sigma\cap\mathcal{U}$ is not a cap, then there exists a line $r$ in $\sigma$ with at least three points of $\mathcal{U}$. By Property 1., $r$ is contained in $\mathcal{U}$.
    By Property $2.$, there exists another line contained in $\mathcal{U}$ through each of the $q$ points on $r$; let $\ell_1,\ldots,\ell_q$ denote those lines. They are necessarily pairwise disjoint since otherwise, we would find a plane with three lines of $\mathcal{U}$. Hence, the $q$ distinct planes $\langle r,\ell_j\rangle$, $j=1,\ldots,q$, intersect $\mathcal{U}$ precisely in $\ell_j$ and $r$, and the lines $\ell_j$ meet $r$ each in a different point. As $|\mathcal{U}|=q^2+q$ (Property $3.$), the remaining plane $\tau$ through $r$ contains precisely $q$ points of $\mathcal{U}$ not on the line $r$.
    Let $Q_1$ and $Q_2$ be two distinct such points.
    If $\langle Q_1,Q_2\rangle$ intersects $r$, then $\langle Q_1,Q_2\rangle$ contains three distinct points of $\mathcal{U}$ and hence, by Property $1.$, is contained in $\mathcal{U}$, which implies that $\langle Q_1,Q_2\rangle\cap r$ is a point of $\mathcal{U}$ through which there exist at least three lines fully contained in $\mathcal{U}$, contradicting Property $2$. We find that the $q$ points of $(\tau\cap\mathcal{U})\setminus r$ are precisely those of an affine line, parallel with $r$ (*).
    
    Let $\mu(\mathcal{U})$ denote the set of projective lines of $\pg{3}{q}$ whose affine points are contained in the set $\mathcal{U}$, and let $\mathcal{U}_\infty$ be the set of points in $\pi_\infty$ which are contained in a line of $\mu(\mathcal{U})$.
    Let $\Tilde{\mathcal{U}}:=\mathcal{U}\cup \mathcal{U}_\infty$. 
    Now we prove that $\Tilde{\mathcal{U}}$, together with the set of projective lines $\mu(\mathcal{U})$, form a generalised quadrangle with parameters $(s,t)=(q,1)$ embedded in $\pg{3}{q}$, and hence, a hyperbolic quadric $Q^+(3,q)$.    As $\mu(\mathcal{U})$ is a set of projective lines, each one contains $q+1=s+1$ points.
    
    Moreover, by Property $2.$, we know that every affine point is contained in precisely $2=t+1$ lines. 
    Hence let $P\in\mathcal{U}_\infty$ be a point at infinity incident with a line $\ell_P\in\mu(\mathcal{U})$. From ($\ast$), we have that there is precisely one line in $\mu(\mathcal{U})$, different from $\ell_P$ whose extension is $P$. 
    Since there are $q^2+q$ points in $\mathcal{U}$, each on exactly $2$ lines, we have that there are $2(q+1)$ lines contained in $\mathcal{U}$, giving rise to $q+1$ points in $\pi_\infty$. Furthermore, it follows from the fact that there are no planes with more than $2$ lines that there are no triangles in $\Tilde{\mathcal{U}}$. Hence, $\Tilde{\mathcal{U}}$ is indeed a generalised quadrangle of order $(q,1)$ embedded in $\pg{3}{q}$. Since it has $q^2+q$ affine points by Proposition 3, it meets $\pi_\infty$ in $q+1$ points forming a non-degenerate conic.
    \end{proof}

\begin{lm}\label{oneway}
    Suppose that $q\geq 5$.
    If $\mS\ni P_\infty$ is a scattered linear set of rank $3$ of $\pg{1}{q^3}$, then the ABB-representation of $\mS\setminus\set{P_\infty}$ is the affine part of a hyperbolic quadric $\mathcal{Q}$ intersecting the plane $\pi_\infty$ in a non-degenerate conic. Furthermore, the extension of this conic contains the $3$ conjugate points defining the spread element $\pi_\infty$.
\end{lm}
\begin{proof}
Let  $\mS\ni P_\infty$ be a point set of $\pg{1}{q^3}$, which is a scattered linear set of rank $3$ and let $T$ be the ABB-representation of $S\setminus\{P_\infty \}$.

We see that the three conditions of Proposition \ref{combprop} hold for $\mathcal{U}=T$:

\begin{enumerate}
    \item An affine line $\ell\in \Pi$ corresponds to a tangent subline of $\pg{1}{q^3}$. Condition $1$ follows from Result \ref{Res_LinearSetIntersectionSubline}.
    \item By Result \ref{Res_SublinesTangent} we know that through every two distinct points $P_1, P_2$ of $S$ there are precisely two $\mathbb{F}_q$-sublines contained in $S$. Let $P_1$ be the point at infinity $P_\infty$ and let $P_2$ be a random affine point in $S$. Then we know that $P_2$ is contained in precisely two tangent $\mathbb{F}_q$-sublines.  Hence, we know by Result $\ref{Res_SublinesTangent}$ that $\varphi(P_2)$ is contained in precisely two lines fully contained in $T$.
    \item The scattered linear set contains $q^2+q+1$ points, of which $q^2+q$ affine ones. 
\end{enumerate}
This implies that $T$ is the affine point set of a hyperbolic quadric. Now consider $\mathcal{Q}$, the extension to $\FF_{q^t}$ of the projective completion of $T$.

By Proposition \ref{Res_LinearSetClub}, through two points of $\mS\setminus\{P_\infty\}$, there are two sublines contained in $\mS$, at least one of which, say $m$, does not contain $P_\infty$. 
By Result \ref{Res_SublinesTangent}, we know that the $\FF_q$-subline $m$, corresponds to a normal rational curve $\mathcal{C}$ whose extension to $\FF_{q^t}$ contains the $3$ conjugate points defining the spread element $\pi_\infty$. Since $m\subseteq\mS$, the extension of $\mathcal{C}$ is contained in $\mathcal{Q}$, and hence, $\mathcal{Q}$ contains the $3$ conjugate points defining $\pi_\infty$.
\end{proof}
\begin{rmk} The first part of Lemma 4.3 can also be proven using the coordinate description of $\mathcal{B}(\pi)$, where $\pi$ is a scattered plane in $\pg{5}{q}$ with respect to the Desarguesian plane spread $\spread$, derived in \cite{LSZ}. If we intersect the hypersurface, whose coordinates are explicitly described there, with a $3$-dimensional subspace containing a spread element $S$ of $\spread$, we find the union of a hyperbolic quadric with the points of $S$. To show that the extension of this hyperbolic quadric contains the $3$ conjugate points, one could then use the coordinates for the indicator sets derived in \cite{BarwickJackson}.
\end{rmk}

\begin{prop}\label{Prop_Numberhyperbol}
    There exists $\frac{1}{2} q^3(q^3-1)$ hyperbolic quadrics $\mathcal{Q}$ in $\Pi$, intersecting the plane $\pi_\infty$ in a non-degenerate conic $\mathcal{C}$ such that its $\mathbb{F}_{q^t}-$extension contains the $3$ conjugate points generated by the spreadelement $\pi_\infty$.
\end{prop}
\begin{proof}
    We again use the fact that all non-degenerate conics in $\pi_\infty$, such that its extension contains three fixed conjugated points, together with all points in $\pi_\infty$ form a $2-(\theta_2, q+1,1)$-design as shown in \cite{BakerBrownEbertFisher}. Hence, there are $\theta_2$ possibilities for choosing an appropriate conic in $\pi_\infty$. It is known that the total number of hyperbolic quadrics in $\Pi$ is $\frac{1}{2}q^4(q^2+1)(q^3-1)$, the number of non-degenerate conics contained in a fixed hyperbolic quadric is $\theta_3-(q+1)^2=q(q^2-1)$ and the number of non-degenerate conics in a solid is $\theta_3 q^2(q^3-1)$ \cite{Hirschfeldgalois}. We can now perform a double counting to obtain that there exist 
    \begin{align*}
        \frac{\frac{1}{2}q^4(q^2+1)(q^3-1)q(q^2-1)}{\theta_3 q^2(q^3-1)}=\frac{1}{2}q^3(q-1)
    \end{align*}
    hyperbolic quadrics containing a fixed non-degenerate conic. Hence, in total, there are $\frac{1}{2} q^3(q^3-1)$ hyperbolic quadrics $\mathcal{Q}$ in $\Pi$, intersecting the plane $\pi_\infty$ in a non-degenerate conic $\mathcal{C}$ such that its $\mathbb{F}_{q^t}-$extension contains the $3$ conjugate points generated by the spreadelement $\pi_\infty$. 
\end{proof}

\begin{prop}\label{Prop_NumberScattered}
    Let $q\geq 5$. There exists $\frac{1}{2}q^3(q^3-1)$ scattered linear sets of rank $3$  in $\pg{1}{q^3}$ which contain $P_\infty$.
\end{prop}
\begin{proof} We will first count the number of scattered planes in $\pg{5}{q}$ with respect to the Desarguesian plane spread $\spread$. There are $\qbin{6}{3}$ planes in $\pg{5}{q}$, of which $q^3+1$ are elements of $\spread$. Now consider  triples $(S,L,\pi)$, where $S$ is an element of $\spread$, $L$ is a line in $S$, and $\pi$ is a plane containing $L$, different from $S$. It easily follows that there are $(q^3+1)(q^2+q+1)(q^3+q^2+q)$ such triples, and since the choice of the plane $\pi$ defines $S$ and $L$ in a unique way, we find $(q^3+1)(q^2+q+1)(q^3+q^2+q)$ planes meeting some spread element in exactly a line. We conclude that there are $\qbin{6}{3}-(q^3+1)-(q^3+1)(q^2+q+1)(q^3+q^2+q)=(q^3+1)q^3(q^3-1)$ scattered planes.
Now count $(\pi,r,S)$ where $r$ is a point of the scattered plane $\pi$ such that $L_\pi$ is the scattered linear set $S$.
On one hand, we have $(q^3+1)q^3(q^3-1)$ scattered planes $\pi$ determining a unique linear set $S$, and $q^2+q+1$ points $r$. On the other hand, by Result \ref{Res_LinearSetClub}(c), we have that given $S$ and $r$, there are exactly $2$ planes $\pi$ through $r$ with $L_\pi=S$. It follows that $|S|(q^2+q+1)2=(q^3+1)q^3(q^3-1)(q^2+q+1)$, and hence, $|S|=\frac{(q^3+1)q^3(q^3-1)}{2}$. The number of scattered linear sets through each of the $q^3+1$ points of $\pg{1}{q^3}$ is a constant, so there are $\frac{q^3(q^3-1)}{2}$ scattered linear sets through $P_\infty$. \end{proof}
\begin{thm} \label{thmscattered} A set $\mathcal{S}$ is the ABB-representation of the affine point set of a scattered linear set of rank $3$ in $\pg{1}{q^3}$, containing $P_\infty$ if and only if it is the affine point set of a hyperbolic quadric intersecting the plane $\pi_\infty$ in a non-degenerate conic $\mathcal{C}$ such that its $\mathbb{F}_{q^t}-$extension contains the $3$ conjugate points generated by the spreadelement $\pi_\infty$.
\end{thm}
\begin{proof} Lemma \ref{oneway} proves that the ABB-representation of the affine point set of a scattered linear set of rank $3$ in $\pg{1}{q^3}$, containing $P_\infty$ is a hyperbolic quadric intersecting the plane $\pi_\infty$ in a non-degenerate conic $\mathcal{C}$ whose extension contains the $3$ conjugate points generating the spreadelement $\pi_\infty$. For the other direction, it suffices to note that the number of such hyperbolic quadrics found in Proposition  \ref{Prop_Numberhyperbol} is precisely the number of scattered linear sets containing $P_\infty$ counted in Proposition \ref{Prop_NumberScattered}.
\end{proof}

%

\section{The optimal case of seven planes of \texorpdfstring{$\pgTitle{5}{q}$}{PG(5,q)} in higgledy-piggledy arrangement}\label{sec5}

%
%

In order to define higgledy-piggledy sets, we need the concept of a {\em strong $k$-blocking set}, which was introduced in \cite[Definition $3.1$]{DavydovEtAl}. They have also appeared in the literature under the terminology \emph{generator sets} and \emph{cutting blocking sets}.

\begin{df}
    Let $k\in\set{0,1,\dots,n-1}$.
    A \emph{strong $k$-blocking set} in $\pg{n}{q}$ is a point set that meets every $(n-k)$-dimensional subspace $\kappa$ in a set of points spanning $\kappa$.
\end{df}


\begin{df}
    Let $k\in\set{0,1,\dots,n-1}$ and suppose that $\mathcal{K}$ is a set of $k$-subspaces in $\pg{n}{q}$.
    If the union of points contained in at least one subspace of $\mathcal{K}$ is a strong $k$-blocking set, then the elements of $\mathcal{K}$ are said to be in \emph{higgledy-piggledy arrangement} and the set $\mathcal{K}$ itself is said to be a \emph{higgledy-piggledy set of $k$-subspaces}.
\end{df}
%
%

The goal is to construct higgledy-piggledy sets of small size. The following particular cases follow from the known lower bounds (see \cite{FancsaliSziklai3}, and \cite{Denaux} for a slight improvement):
%

\begin{crl}\label{Crl_LowerBound}
    If $0<k<n-1$ and $q\geq 7$, then a higgledy-piggledy set of $k$-subspaces
    \begin{enumerate}
        \item contains at least $4$ elements if $n=3$,
        \item contains at least $6$ elements if $n=4$, and
        \item contains at least $7$ elements if $n=5$.
    \end{enumerate}
\end{crl}

The above lower bounds are sharp (\cite[Theorem $3.7$, Example $9$]{DavydovEtAl,FancsaliSziklai2}, \cite[Proposition $12$]{BartoliKissMarcuginiPambianco}, \cite[Theorem $3.15$]{BartoliCossidenteMarinoPavese}, \cite[Theorem $33$ and $39$, Corollary $34$ and $35$]{Denaux}), except for the case $(n,k)=(5,2)$. Concerning the latter case, the author of \cite{Denaux} used the following construction to find $8$ planes in higgledy-piggledy arrangment.
%
%

\begin{crl}\label{Crl_ConstructionMethod}
    Suppose that $\mathcal{P}$ is a point set of $\pg{1}{q^3}$ that is not contained in any $\FF_q$-linear set of rank at most $3$.
    Then $\mathcal{F}(\mathcal{P})$ is a higgledy-piggledy set of pairwise disjoint planes in $\pg{5}{q}$.
\end{crl}
\begin{proof}
    This is a special case of \cite[Theorem $16$]{Denaux}.
\end{proof}

Any higgledy-piggledy set of planes constructed in this way consists of disjoint planes; however, it is worth noting that this is not a restriction:
\begin{prop}[{\cite[Proposition $40$]{Denaux}}]
    If $q\geq 7$, then any seven planes of $\pg{5}{q}$ in higgledy-piggledy arrangement are pairwise disjoint.
\end{prop}

Using the results obtained in previous sections, we are able to show that the lower bound of Corollary \ref{Crl_LowerBound} is sharp in the case $n=5$:

\begin{thm}
    There exist seven planes of $\pg{5}{q}$ in higgledy-piggledy arrangement.
\end{thm}
\begin{proof}
If $q\leq5$, we can easily verify the statement using a computer package such as GAP (see e.g.\ \cite[Code Snippet $56$]{Denaux})\footnote{In fact, using similar code, one can check that there exist in fact $6$ planes of $\pg{5}{3}$ and $5$ planes of $\pg{5}{2}$ in higgledy-piggledy arrangement.}.
Hence, assume that $q\geq5$ for the remainder of this proof.
By Corollary \ref{Crl_ConstructionMethod}, it is sufficient to pick $7$ points in $\pg{1}{q^3}$ such that no linear set of rank at most $3$ contains all these $7$ points. First note that if $7$ points are contained in a linear set of rank $<3$, they are also contained in a linear set of rank $3$. Hence, we only need to show that it is possible to pick $7$ points, not contained in a linear set of rank $3$.

Pick a point $P_\infty$ in $\pg{1}{q^3}$. Then we know from Proposition \ref{heads2} that there are $q^3+q^2+q$ clubs with head $P_\infty$, from Proposition \ref{Prop_NumberClubsHeadNotInftyevenmoregeneral} that there are $q^3(q^2+q+1)$ clubs through $P_\infty$ with head different from $P_\infty$, and from Proposition \ref{Prop_NumberScattered} that there are $\frac{1}{2}q^3(q^3-1)$ scattered linear sets containing $P_\infty$.

We will count the set $S=\{(P_1,P_2,P_3,P_4,P_5,P_6,L)\}$ where $P_i\neq P_\infty$ are different points of $\pg{1}{q^3}$ and $L$ is a linear set of rank $3$ containing $P_\infty$ and $P_i$, $i=1,\ldots,6$.
We have that $$|S|=(q^3+q^2+q)c+q^3(q^2+q+1)c+\frac{1}{2}q^3(q^3-1)d,$$ where $c=q^2(q^2-1)(q^2-2)(q^2-3)(q^2-4)(q^2-5)$ is the number of ways to pick $6$ different points different from $P_\infty$ in a club through $P_\infty$, and $d=(q^2+q)(q^2+q-1)(q^2+q-2)(q^2+q-3)(q^2+q-4)(q^2+q-5)$ is the number of ways to pick $6$ points different from $P_\infty$ in a scattered linear set through $P_\infty$.

If all choices of $6$ points $P_1,\ldots,P_6$ would be contained in at least one linear set of rank $3$ through $P_\infty$, then $|S|\geq q^3(q^3-1)(q^3-2)(q^3-3)(q^3-4)(q^3-5)$, a contradiction for $q\geq 3$.

\end{proof}

We will now use the results of this paper to explicitely construct a set of $7$ planes in $\pg{5}{q}$ in higgledy-piggledy arrangement. We start by writing down explicit equations of the set of conics in $\pg{2}{q}$ containing $3$ fixed conjugate points.

\begin{lm}\label{juistevorm} Let $\omega\in \FF_{q^3}\setminus\FF_q$ be a generator of $(\FF_{q^3}^*,.)$ satisfying $\omega^3+\lambda_1\omega^2+\lambda_2\omega+\lambda_3=0$. Then the conics in $\pg{2}{q}$ whose extension to $\FF_{q^3}$ contains the points $(1,\omega,\omega^2)$, $(1,\omega^q,\omega^{2q})$, $(1,\omega^{q^2},\omega^{2q^2})$ are given by
\begin{align}
g_{d,e,f}(X_0,X_1,X_2):=(\lambda_3e-\lambda_1\lambda_3f)X_0^2+(\lambda_2e+(\lambda_3-\lambda_1\lambda_2)f)X_0X_1&+\nonumber\\(\lambda_1e+(\lambda_2-\lambda_{1}^2)f-d)X_0X_2+dX_1^2+eX_1X_2+fX_2^2&=0,\label{specialconics}\end{align}
with $d,e,f\in \FF_q$ not all zero. 
\end{lm}
\begin{proof} An arbitrary conic $\mathcal{C}$ in $\pg{2}{q}$ has equation $aX_0^2+bX_0X_2+cX_0X_2+dX_1^2+eX_1X_2+fX_2^2=0$ where $a,b,c,d,e,f\in \FF_q$. Note that if $(1,\omega,\omega^2)$ lies on the extension of $\mathcal{C}$ to $\pg{2}{q^3}$, then $(1,\omega^q,\omega^{2q})$ and $(1,\omega^{q^2},\omega^{2q^2})$ also lie on this extension. Expressing that $(1,\omega,\omega^2)$ lies on $\mathcal{C}$, using that $\omega^4=(\lambda_1^2-\lambda_2)\omega^2+(\lambda_1\lambda_2-\lambda_3)\omega+\lambda_1\lambda_3$, and that $1,\omega,\omega^2$ are $\FF_q$-independent, we find the following system of equations:
\begin{align*} a-\lambda_3e+\lambda_1\lambda_3f&=0\\
b-\lambda_2e+(\lambda_1\lambda_2-\lambda_3)f&=0\\
c+d-\lambda_1e+(\lambda_1^2-\lambda_2)f&=0.\qedhere
\end{align*}
\end{proof}


\begin{prop} \label{construction} Let $P_i(x_0^{(i)},x_1^{(i)},x_2^{(i)},1)$, $i=1,\ldots, 6$ be six non-coplanar points contained in a non-degenerate elliptic quadric intersecting the plane $\pi:X_3=0$ in the conic $X_0X_2-X_1^2=0$. Consider the quadrics  \begin{align} \mathcal{Q}(d,e,f,u,v,w,t,X_0,X_1,X_2,X_3):=g_{d,e,f}(X_0,X_1,X_2)+X_3(uX_0+vX_1+wX_2+tX_3)=0. \label{quadrics}\end{align} 
Let $A$ be the $(6\times 7)$-matrix whose $i$-th row $(A)_i$ satisfies $$(A)_i[d,e,f,u,v,w,t]^T=\mathcal{Q}(d,e,f,u,v,w,t,x_0^{(i)},x_1^{(i)},x_2^{(i)},1).$$ If $rk(A)=6$, then the points $P_1,\ldots,P_6$, together with $P_\infty$, are the ABB-representation of a set of seven points in $\pg{1}{q^3}$ such that, under field reduction, these seven points form a higgledy-piggledy set of $7$ planes in $\pg{5}{q}$. That is, $\{\mathcal{F}(\phi^{-1}(P_i))\mid 1\leq i\leq 6\}\cup \mathcal{F}(P_\infty)$ is a set of seven planes in $\pg{5}{q}$ in higgledy-piggledy arrangement.
\end{prop}

\begin{proof} By Corollary \ref{Crl_ConstructionMethod}, it is sufficient to construct a set of $7$ points in $\pg{1}{q^3}$ such that no linear set of rank at most $3$ contains all these $7$ points. Embed the line $L=\pg{1}{q^3}$ in $\pg{2}{q^3}$ and select one point $P_\infty$ on $L$. Let $\ell_\infty$ be a line of $\pg{2}{q^3}$ through $P_\infty$, different from $L$ and consider the ABB-representation of $\pg{2}{q^3}$ with $\ell_\infty$ as line at infinity. Then the set of points $\mathcal{F}(P)$, with $P$ a point of $L$ different from $P_\infty$, defines a $3$-dimensional subspace $\Pi$. We coordinatise in such way that the points in $\Pi$ have coordinates $(x_0,x_1,x_2,x_3)$ such that the points with $x_3=0$ are the points in the plane $\pi=\mathcal{F}(P_\infty)$ and the three conjugate points defining $\pi$ are $(1,\omega,\omega^2)$, $(1,\omega^q,\omega^{2q})$, $(1,\omega^{q^2},\omega^{2q^2})$.
In view of Proposition \ref{trivial}, Theorem \ref{thmclub}, and Theorem \ref{thmscattered}, we need to find six affine points of $\Pi$ such that these are not contained in a plane, nor a cone with vertex not in $\pi$ and base a conic whose extension contains the $3$ conjugate points, nor a hyperbolic quadric through such a conic. All (possibly degenerate) quadrics meeting in a conic of the form \eqref{specialconics} are given by an equation of the form \begin{align} f_{d,e,f}(X_0,X_1,X_2)+X_3(uX_0+vX_1+wX_2+tX_3)=0. \label{quadricsform}\end{align}
So if we pick six points, contained in an elliptic quadric $\mathcal{E}$ meeting $\pi$ in the conic  $X_0X_2-X_1^2=0$, we simply need to show that $\mathcal{E}$ is the only quadric with equation of the form \eqref{quadricsform} through those $6$ points. This happens if and only if the homogeneous system of $6$ equations in the variables $d,e,f,u,v,w,t$ that arises from substituting the coordinates of the six points has a unique solution up to scalar multiple, which happens if and only if its coefficient matrix $A$ has $rk(A)=6$.
\end{proof}

\bigskip
In order to give an explicit construction of six such points and make the computations easier, we will restrict ourselves to those values of $q$ such that there is a primitive cubic polynomial of a particular form.

\begin{thm} \begin{itemize} \item[(a)] Let $q$ be odd, $q\equiv 1\pmod{3}$. Let $a$ be a non-square in $\FF_q$, where $a\neq \frac{1}{2}$. The six points $(1,0,-a,1)$,$(1,0,-a,-1)$,$(1,1,1-a,1)$,$(1,-1,1-a,1)$,$(1,1,1-a,-1)$, $(1,-1,1-a,-1)$ give rise to a higgledy-piggledy set of $7$ planes in $\pg{5}{q}$.
\item[(b)] Let $q$ be even such that there is an irreducible polynomial of the form $\omega^3+\omega+1=0$. Let $a\in \FF_q$ with $Tr(a)=1$, $a\neq 1$. The six points $(1,0,a,1)$,$(1,1,a,1)$,$(a,0,1,1)$, $(a,1,1,1)$, $(1,a,a^2,1)$, $(a^2,a,1,1)$ give rise to a higgledy-piggledy set of $7$ planes in $\pg{5}{q}$.\end{itemize}
\end{thm}
\begin{proof}\begin{itemize}\item[(a)] Since $q\equiv 1\pmod{3}$, there is an irreducible polynomial of the form $\omega^3+\lambda=0$. Using Lemma \ref{juistevorm}, we find that the quadrics of the form \eqref{quadrics}  become
\begin{align} \lambda eX_0^2+\lambda fX_0X_1-dX_0X_2+dX_1^2+eX_1X_2+fX_2^2+X_3(uX_0+vX_1+wX_2+tX_3)=0.\end{align}
It is easy to check that the given six points are not coplanar. Furthermore, they are contained in 
 the elliptic quadric $\mathcal{E}$ with equation $X_0X_2-X_1^2-aX_3^2=0$, which meets $\pi$ in the conic $X_0X_2-X_1^2=0$.  Substituting the $6$ points into \eqref{quadrics} yields a system $\Xi$ of $6$ homogeneous equations in $d,e,f,u,v,w,t$ whose associated coefficient matrix is given by 
 $$\begin{bmatrix} a &\lambda&a^2&1&0&-a&1\\
 a &\lambda&a^2&-1&0&a&1\\
 a &\lambda+1-a&(1-a)^2+\lambda&1&1&1-a&1\\
  a &\lambda+a-1&(1-a)^2-\lambda&1&-1&1-a&1\\
   a &\lambda+1-a&(1-a)^2+\lambda&-1&-1&a-1&1\\
      a &\lambda+a-1&(1-a)^2-\lambda&-1&1&a-1&1
 \end{bmatrix}$$
 
 It can be checked that this matrix has full rank if and only if $a(1-a)(2a-1)\neq 0$. The statement follows from Proposition \ref{construction}.

\item[(b)] Now assume that $q$ is even and $\omega^3=\omega+1$. 
Using Lemma \ref{juistevorm}, we find that the equation for the quadrics \eqref{quadrics} now becomes
\begin{align} eX_0^2+(e+ f)X_0X_1+(d+f)X_0X_2+dX_1^2+eX_1X_2+fX_2^2\\ +X_3(uX_0+vX_1+wX_2+tX_3)=0.\label{qu2}\end{align}
The six given points are contained in the elliptic quadric $\mathcal{E}$ with equation $X_0X_2+X_1^2+X_1X_3+aX_3^2=0$, which meets $\pi$ in $X_0X_2+X_1^2=0$.
Again, these points are not coplanar, and expressing that those six points lie on an equation of the form \eqref{qu2} yields a system $\Xi$ in $d,e,f,u,v,w,t$ with coefficient matrix

$$\begin{bmatrix} a&1&a+a^2&1&0&a&1\\
1+a&a&1+a+a^2&1&1&a&1\\
a&a^2&a+1&a&0&1&1\\
1+a&a^2+a+1&1&a&1&1&1\\
0&1+a+a^3&a+a^2+a^4&1&a&a^2&1\\
0&a^4+a^3+a&a^3+a^2+1&a^2&a&1&1\end{bmatrix}$$
This matrix has full rank if and only if $a(1+a)\neq 0$. Hence, since $a\neq 0,1$, the statement follows from Proposition \ref{construction}.\qedhere
\end{itemize}
\end{proof}

\bibliographystyle{plain}
\bibliography{main.bib}

\begin{thebibliography}{10}

\bibitem{sam}
S.~Adriaensen and L.~Denaux.
\newblock Small weight codewords of projective geometric codes.
\newblock {\em J. Combin. Theory Ser. A}, 180:Paper No. 105395, 34, 2021.

\bibitem{Andre}
J.~Andr\'{e}.
\newblock \"{U}ber nicht-{D}esarguessche {E}benen mit transitiver
  {T}ranslationsgruppe.
\newblock {\em Math. Z.}, 60:156--186, 1954.

\bibitem{BakerBrownEbertFisher}
R.~D. Baker, J.~M.~N. Brown, G.~L. Ebert, and J.~C. Fisher.
\newblock Projective bundles.
\newblock {\em Bull. Belg. Math. Soc. Simon Stevin}, 1(3):329--336, 1994.
\newblock A tribute to J. A. Thas (Gent, 1994).

\bibitem{BartoliCossidenteMarinoPavese}
D.~Bartoli, A.~Cossidente, G.~Marino, and F.~Pavese.
\newblock On cutting blocking sets and their codes.
\newblock {\em Forum Math.}, 34(2):347--368, 2022.

\bibitem{BartoliKissMarcuginiPambianco}
D.~Bartoli, G.~Kiss, S.~Marcugini, and F.~Pambianco.
\newblock Resolving sets for higher dimensional projective spaces.
\newblock {\em Finite Fields Appl.}, 67:101723, 14, 2020.

\bibitem{ABB1}
S.~G. Barwick, L.~R.~A. Casse, and C.~T. Quinn.
\newblock The {A}ndr\'{e}/{B}ruck and {B}ose representation in {${\rm
  PG}(2h,q)$}: unitals and {B}aer subplanes.
\newblock {\em Bull. Belg. Math. Soc. Simon Stevin}, 7(2):173--197, 2000.

\bibitem{BarwickJackson}
S.~G. Barwick and W.~Jackson.
\newblock Sublines and subplanes of {${\rm PG}(2,q^3)$} in the {B}ruck-{B}ose
  representation in {${\rm PG}(6,q)$}.
\newblock {\em Finite Fields Appl.}, 18(1):93--107, 2012.

\bibitem{BruckBose}
R.~H. Bruck and R.~C. Bose.
\newblock The construction of translation planes from projective spaces.
\newblock {\em J. Algebra}, 1:85--102, 1964.

\bibitem{IndicatorSet}
L.~R.~A. Casse and C.~M. O'Keefe.
\newblock Indicator sets for {$t$}-spreads of {${\rm PG}((s+1)(t+1)-1,\;q)$}.
\newblock {\em Boll. Un. Mat. Ital. B (7)}, 4(1):13--33, 1990.

\bibitem{equivalence}
B.~Csajb\'{o}k, G.~Marino, and O.~Polverino.
\newblock Classes and equivalence of linear sets in {${\rm PG}(1, q^n)$}.
\newblock {\em J. Combin. Theory Ser. A}, 157:402--426, 2018.

\bibitem{DavydovEtAl}
A.~A. Davydov, M.~Giulietti, S.~Marcugini, and F.~Pambianco.
\newblock Linear nonbinary covering codes and saturating sets in projective
  spaces.
\newblock {\em Adv. Math. Commun.}, 5(1):119--147, 2011.

\bibitem{Denaux}
L.~Denaux.
\newblock Higgledy-{P}iggledy {S}ets in {P}rojective {S}paces of {S}mall
  {D}imension.
\newblock {\em Electron. J. Combin.}, 29(3):Paper No. 3.29--, 2022.

\bibitem{cyclicmodel}
G.~Faina, G.~Kiss, S.~Marcugini, and F.~Pambianco.
\newblock The cyclic model for {${\rm PG}(n,q)$} and a construction of arcs.
\newblock {\em European J. Combin.}, 23(1):31--35, 2002.

\bibitem{clubs}
Sz.~L. Fancsali and P.~Sziklai.
\newblock Description of the clubs.
\newblock {\em Ann. Univ. Sci. Budapest. E\"{o}tv\"{o}s Sect. Math.},
  51:141--146 (2009), 2008.

\bibitem{FancsaliSziklai2}
Sz.~L. Fancsali and P.~Sziklai.
\newblock Lines in higgledy-piggledy arrangement.
\newblock {\em Electron. J. Combin.}, 21(2):Paper 2.56, 15, 2014.

\bibitem{FancsaliSziklai3}
Sz.~L. Fancsali and P.~Sziklai.
\newblock Higgledy-piggledy subspaces and uniform subspace designs.
\newblock {\em Des. Codes Cryptogr.}, 79(3):625--645, 2016.

\bibitem{Harris}
J.~Harris.
\newblock {\em Algebraic geometry}, volume 133 of {\em Graduate Texts in
  Mathematics}.
\newblock Springer-Verlag, New York, 1995.
\newblock A first course.

\bibitem{Hirschfeldgalois}
J.~W.~P. Hirschfeld and J.~A. Thas.
\newblock {\em General {G}alois geometries}.
\newblock Oxford Mathematical Monographs. The Clarendon Press, Oxford
  University Press, New York, 1991.
\newblock Oxford Science Publications.

\bibitem{LSZ}
M.~Lavrauw, J.~Sheekey, and C.~Zanella.
\newblock On embeddings of minimum dimension of {${\rm PG}(n,q)\times{\rm
  PG}(n,q)$}.
\newblock {\em Des. Codes Cryptogr.}, 74(2):427--440, 2015.

\bibitem{LavrauwVandeVoordeLinSetLine}
M.~Lavrauw and G.~Van~de Voorde.
\newblock On linear sets on a projective line.
\newblock {\em Des. Codes Cryptogr.}, 56(2-3):89--104, 2010.

\bibitem{LavrauwVandeVoordeFieldRed}
M.~Lavrauw and G.~Van~de Voorde.
\newblock Field reduction and linear sets in finite geometry.
\newblock In {\em Topics in finite fields}, volume 632 of {\em Contemp. Math.},
  pages 271--293. Amer. Math. Soc., Providence, RI, 2015.

\bibitem{lunardon}
G.~Lunardon and O.~Polverino.
\newblock Translation ovoids of orthogonal polar spaces.
\newblock {\em Forum Math.}, 16(5):663--669, 2004.

\bibitem{NPSZ}
V.~Napolitano, O.~Polverino, P.~Santonastaso, and F.~Zullo.
\newblock Clubs and their applications, arxiv: 2209.13339, 2022.

\bibitem{olga}
O.~Polverino.
\newblock Linear sets in finite projective spaces.
\newblock {\em Discrete Math.}, 310(22):3096--3107, 2010.

\bibitem{olgaZullo}
O.~Polverino and F.~Zullo.
\newblock Connections between scattered linear sets and {MRD}-codes.
\newblock {\em Bull. Inst. Combin. Appl.}, 89:46--74, 2020.

\bibitem{ABBconics}
C.~T. Quinn.
\newblock The {A}ndr\'{e}/{B}ruck and {B}ose representation of conics in {B}aer
  subplanes of {${\rm PG}(2,q^2)$}.
\newblock {\em J. Geom.}, 74(1-2):123--138, 2002.

\bibitem{RotteySheekeyVandeVoorde}
S.~Rottey, J.~Sheekey, and G.~Van~de Voorde.
\newblock Subgeometries in the {A}ndr\'{e}/{B}ruck-{B}ose representation.
\newblock {\em Finite Fields Appl.}, 35:115--138, 2015.

\bibitem{Segre64}
B.~Segre.
\newblock Teoria di {G}alois, fibrazioni proiettive e geometrie non
  desarguesiane.
\newblock {\em Ann. Mat. Pura Appl. (4)}, 64:1--76, 1964.

\bibitem{mijnthesis}
G.~Van~de Voorde.
\newblock {\em Blocking Sets in Finite Projective Spaces and Coding Theory}.
\newblock PhD thesis, Universiteit Gent, Belgium, 2010.

\end{thebibliography}

Lins Denaux \& Jozefien D'haeseleer

Ghent University\\
Department of Mathematics: Analysis, Logic and Discrete Mathematics

Krijgslaan $281$ -- Building S$8$

$9000$ Ghent

BELGIUM

\texttt{e-mail: lins.denaux@ugent.be}\\
\texttt{e-mail: jozefien.dhaeseleer@ugent.be}

\texttt{website: }\url{https://users.ugent.be/~ldnaux}\\
\texttt{website: }\url{https://users.ugent.be/~jmdhaese}

\bigskip
Geertrui Van de Voorde

University of Canterbury
(Te Whare W$\overline{\mbox{a}}$nanga o Waitaha)\\
School of Mathematics and Statistics

Private Bag $4800$ -- Erskine Building

Christchurch $8140$

NEW ZEALAND

\texttt{e-mail: geertrui.vandevoorde@canterbury.ac.nz}


\end{document}